\documentclass[10pt]{amsart}

\usepackage{graphicx}
\usepackage{amsfonts,amsmath,latexsym,amssymb,amsthm,enumerate}
\usepackage{hyperref}

\newcounter{theoremcounter}

\newcounter{dummycounter}

\newcounter{emptycounter}

\newtheorem{theorem}[theoremcounter]{Theorem}

\newtheorem{lemma}[theoremcounter]{Lemma}

\newtheorem{proposition}[theoremcounter]{Proposition}
\newtheorem{corollary}[theoremcounter]{Corollary}

\newtheorem{definition}[theoremcounter]{Definition}

\numberwithin{equation}{section}
\numberwithin{lemmacounter}{section}
\numberwithin{propcounter}{section}
\numberwithin{corcounter}{section}
\numberwithin{conjcounter}{section}
\numberwithin{theoremcounter}{section}
\numberwithin{probcounter}{section}

\newcounter{eqncounter}

\numberwithin{equation}{eqncounter}

\def\IR{\mathbb R}

\def\IZ{\mathbb Z}
\def\IN{\mathbb N}

\def\IP{\mathbb P}
\def\IQ{\mathbb Q}

\def\P{\mathcal{P}}

\def\diam{\text{diam}}

\def\H{{\mathcal{H}}}
\def\Leb{{\mathcal{L}}}
\def\L{B}
\def\M{E}
\def\p{p}
\def\a{a}

\def\Vol{\textup{Vol}}
\def\cl{\textup{cl}}
\def\bd{\textup{bd}}
\def\dim{\textup{dim }}

\def\inte{\textup{int}}
\def\D{{D}}

\def\pit{\widetilde{\pi}}

\newcommand{\N}{\mathbb{N}}
\newcommand{\R}{\mathbb{R}}
\newcommand{\Z}{\mathbb{Z}}
\newcommand{\Q}{\mathbb{Q}}

\newcommand{\Ppsi}{\Psi}

\newcommand{\tx}[1]{\text{#1}}

\newcommand{\de}{\mbox{d}}

\begin{document}\baselineskip=17pt

\author{Fabrizio Barroero}
\address{Institute of Analysis and Computational Number Theory (Math A),
Graz University of Technology,
Steyrergasse 30, A-8010 Graz,
Austria}
\email{barroero@math.tugraz.at}
\thanks{F. Barroero is supported by the Austrian Science Foundation (FWF) project W1230-N13.}

\author{Martin Widmer}
\address{Scuola Normale Superiore di Pisa,
56126 Pisa, Italy}
\email{martin.widmer@sns.it}
\thanks{M. Widmer was supported in part by the Austrian Science Foundation (FWF) project M1222-N13 and ERC-Grant No. 267273.}

\title{Counting lattice points and o-minimal structures}

\date{\today}



\begin{abstract}
Let $\Lambda$ be a lattice in $\R^n$, and let $Z\subseteq \R^{m+n}$ be a definable family in an o-minimal structure over $\R$.
We give sharp estimates for the number of lattice points in the fibers $Z_T=\{x\in \R^n: (T,x)\in Z\}$.
Along the way we show that for any subspace $\Sigma\subseteq\R^n$ of dimension $j>0$ the $j$-volume of the orthogonal projection of $Z_T$ to
$\Sigma$ is, up to a constant depending only on the family $Z$, bounded by the maximal $j$-dimensional volume of the orthogonal projections to the $j$-dimensional coordinate subspaces. 
\end{abstract}

\maketitle

\section{Introduction}\label{intro}
Let $\Lambda$ be a lattice in $\IR^n$, and let $Z$ be a subset of $\R^{m+n}$. We consider $Z$ as a parameterized family of subsets 
$Z_T=\{x\subseteq \IR^n: (T,x)\in Z\}$ of $\IR^n$.
One is often led to the problem of estimating the cardinality $|\Lambda\cap Z_T|$ as the parameter $T$
ranges over an infinite set. According to a general principle one would expect that,
if the sets $Z_T$ are reasonably shaped, a good estimate for $|\Lambda\cap Z_T|$ is given by
$\Vol(Z_T)/\det\Lambda$. The situation is relatively easy if
$Z_T=T Z_1$ for some fixed subset $Z_1$ of $\IR^n$ and as $T\in \IR$ tends to infinity.\footnote{However, even if $Z_T=TZ_1$ is compact
it is not necessarily true that $|\Lambda\cap Z_T|=\Vol(Z_1)T^n/\det\Lambda+O(T^{n-1})$, e.g., 
take $\Lambda=\IZ^n$, and $Z_1=\{0,2^{-1},2^{-2},2^{-3},\ldots\}\times [0,1]^{n-1}$.
The latter is a counterexample to the claim in the first paragraph of \cite{15}.} 
However,
in many situations the family $Z$ is more complicated, and typically described by inequalities such as
\begin{alignat}1\label{Zfunc}
f_1(T_1,\ldots,T_m,x_1,\ldots,x_n)\leq 0,\ldots, f_N(T_1,\ldots,T_m,x_1,\ldots,x_n)\leq 0,
\end{alignat}
where the $f_i$ are certain real valued functions on $\IR^{m+n}$, e.g., polynomials.
Using the language of o-minimal structures from model theory we prove for fairly general
families $Z$ an estimate for $|\Lambda\cap Z_T|$, which is quite precise in terms
of the geometry of the sets $Z_T$, and the geometry of the lattice $\Lambda$.

A classical result, although restricted to $\Lambda=\IZ^n$, was proven by Davenport \cite[Theorem]{15}. 
\begin{theorem}[Davenport] \label{thmDav}
Let $n$ be a positive integer, and let $Z_T$ be a compact set in $\IR^{n}$ that satisfies the following conditions.
\begin{enumerate}
\item Any line parallel to one of the $n$ coordinate axes intersects $Z_T$ in a set of points, which, if not empty, consists of at most $h$ intervals.
\item The same is true (with $j$ in place of $n$) for any of the $j$ dimensional regions obtained by orthogonally projecting $Z_T$ on one of the 
coordinate spaces defined by equating a selection of $n-j$ of the coordinates to zero, and this condition is satisfied for all $j$ from 1 to $n-1$.
\end{enumerate}
Then
\begin{alignat*}3
\left| \left|Z_T\cap \IZ^n \right|-\Vol(Z_T) \right| \leq \sum_{j=0}^{n-1} h^{n-j} V_j(Z_T),
\end{alignat*}
where $V_j(Z_T)$ is the sum of the $j$-dimensional volumes of the orthogonal projections of $Z_T$ on 
the various coordinate spaces obtained by equating any $n-j$ coordinates to zero, and $V_0(Z_T)=1$ by convention.
\end{theorem}
A drawback of Davenport's theorem is that the conditions (1) and (2) are often difficult to verify. 
Various authors have given similar estimates for general lattices with simpler, possibly milder, conditions on the set; see \cite{WidmerLNCL} 
for a discussion on that. 
Classical results are known for homogeneously expanding sets whose boundary is parameterizable by certain Lipschitz maps, see, e.g., \cite[Theorem 5.1, Chap. 3]{3},
or \cite[Theorem]{Spain} for a refined version.  
Masser and Vaaler \cite[Lemma 2]{1} gave a counting result for sets satisfying the above Lipschitz condition but which are not necessarily homogeneously expanding,
and moreover, the dependence on the lattice was made explicit.
Masser and Vaaler's result was refined by the second author \cite[Theorem 5.4]{art1} to get a sharp error term  (for balls such sharp estimates have been
obtained by Schmidt in \cite[Lemma 2]{36}).
However, all these results for general lattices have one drawback in common:
usually, a direct application yields nontrivial estimates
only if the volume is much larger than the diameter; e.g., if $T\in \R$ tends to infinity we usually require 
$\diam(Z_T)^{n-1}=o(\Vol(Z_T))$.  We shall illustrate this problem more
explicitly after we have stated our theorem.

Of course, Davenport's theorem can easily be generalized to arbitrary lattices. With a bit care, using standard results from Geometry of Numbers, 
one gets the error term (ignoring a factor depending only on $n$)
\begin{alignat}1\label{Thunderbound}
\sum_{j=0}^{n-1}h'(Z_T)^{n-j}\frac{V_j'(Z_T)}{\lambda_1\cdots \lambda_j}, 
\end{alignat}
where $\lambda_1,\ldots,\lambda_n$ are the successive minima of $\Lambda$ (with respect to the zero-centered unit ball), 
$V_j'(Z_T)$ is the supremum of the volumes of the orthogonal 
projections of $Z_T$ to the $j$-dimensional linear subspaces, and $h'$ is what we get instead of $h$ when
in Davenport's conditions {\it ``line parallel to one of the $n$ coordinate axes''} and 
{\it ``orthogonally projecting $Z_T$ on one of the 
coordinate spaces defined by equating a selection of $n-j$ of the coordinates to zero''}
are replaced by {\it ``line''} and {\it ``any projection of $Z_T$ on any $j$-dimensional subspace''}.
 
Now the quantity $V_j'(Z_T)$ is definitely not so nice to work with as $V_j(Z_T)$.
Moreover, proving the existence of uniform upper bounds for $h'(Z_T)$ (i.e., independent of $T$) is often troublesome and awkward. 
Therefore it would be nice to have some general but mild conditions on the family $Z$ that allow us to replace $h'(Z_T)$ by a uniform constant $c_Z$
and $V_j'(Z_T)$ by $V_j(Z_T)$. 

At this point it might be worthwhile to emphasize that even if the sets $Z_T$ are simply given by a finite number of squares in $\R^2$ we cannot
expect that $V_j'(Z_T)\leq c V_j(Z_T)$ for some absolute constant $c$; consider the sets $C_n\times C_n$ in \cite[Example 2.67]{AmbrosioFuscoPallara} for a simple
counterexample. The latter example indicates that such an inequality would require a rather strong hypothesis on the family $Z$.
Also, to handle $h'$ we need that the number of connected components of a projection of $Z_T$ when intersected with a line is uniformly bounded.

The setting of o-minimal structures delivers exactly the required topological properties, and therefore seems to be the natural framework suitable
for our problem. Furthermore, it provides a rich and flexible structure, including many of the relevant examples. 

We are using the notation of \cite{vandenDries} and \cite{15}. We write $\IN=\{1,2,3,\ldots\}$ for the set of positive integers.

\begin{definition}\label{defomin}
An o-minimal structure is a sequence $\mathcal{S}=(\mathcal{S}_n)_{n\in \IN}$ of families of subsets in $\R^n$ such that for each $n$:
\begin{enumerate}
\item $\mathcal{S}_n$ is a boolean algebra of subsets of $\IR^n$, that is, $\mathcal{S}_n$ is a collection of subsets of $\IR^n$,
$\emptyset \in \mathcal{S}_n$, and if $A, B \in \mathcal{S}_n$ then also $A\cup B \in \mathcal{S}_n$, and $\IR^n\backslash A\in \mathcal{S}_n$. 
\item If $A \in \mathcal{S}_n$ then  $\IR\times A \in \mathcal{S}_{n+1}$ and $A\times \IR \in \mathcal{S}_{n+1}$.
\item $\{(x_1,\ldots,x_n): x_i=x_j\}\in \mathcal{S}_n$ for $1\leq i<j\leq n$.
\item If $\pi: \IR^{n+1}\rightarrow \IR^n$ is the projection map on the first $n$ coordinates and $A \in \mathcal{S}_{n+1}$ then $\pi(A) \in \mathcal{S}_n$.
\item $\{r\}\in \mathcal{S}_1$ for any $r\in \IR$ and $\{(x,y)\in \IR^2: x<y\}\in \mathcal{S}_2$.
\item The only sets in $\mathcal{S}_1$ are the finite unions of intervals and points. (``Interval'' always means ``open interval'' with infinite endpoints allowed.) 
\end{enumerate}
\end{definition} 
Following the usual convention, we say a set $A$ is definable (in $\mathcal{S}$) if it lies in some $\mathcal{S}_n$.\\

Next we give some important examples of o-minimal structures, following the presentation of Scanlon in \cite{Scanlon1037}. 
For each $n\in \IN$ let $F_n$ be a collection of functions $f:\R^n\rightarrow \R$ that we call distinguished functions.
If $g,h:\R^n\rightarrow \R$ are built from the coordinate functions, constant functions and distinguished functions by composition 
(provided it is defined), then we say
\begin{alignat*}1
\{x\in \R^n: g(x)<h(x)\},\\
\{x\in \R^n: g(x)=h(x)\},
\end{alignat*}
are atomic sets. Now let us consider the smallest family of sets in $\R^n$ (for various $n$) that contains all atomic sets, and is closed under finite unions and complements, 
and images of the usual projection maps $\pi:\R^{n+1}\rightarrow \R^n$ onto the first $n$ coordinates.
For the following choices of  $F=\bigcup_n F_n$, the resulting family consists precisely of the definable 
sets in a particular o-minimal structure: 
\begin{enumerate}
\item $F_{\tx{alg}}=\{$polynomials defined over $\R\}$,
\item $F_{\tx{an}}=F_{\tx{alg}} \cup \{$restricted analytic functions$\}$,
\item $F_{\tx{exp}}=F_{\tx{alg}} \cup\{$the exponential function $\exp:\R\rightarrow \R\}$,
\item $F_{\tx{an,exp}}=F_{\tx{an}} \cup F_{\tx{exp}}$.
\end{enumerate}
By a restricted analytic function we mean a function $f:\R^n\rightarrow \R$, which is zero outside of $[-1,1]^n$, and is the restriction to $[-1,1]^n$ of a 
function, which is real analytic on an open neighborhood of  $[-1,1]^n$.

For the first example note that by the Tarski-Seidenberg theorem every 
set in this family is a boolean combination of atomic sets, and thus is semialgebraic. 
This implies (6) in Definition \ref{defomin}, and (1)-(5) are clear. The o-minimality of example (2) is due to Denef and van den 
Dries \cite{DenefvandenDries}, while (3) is due to Wilkie \cite{Wilkie96}. Van den Dries and Miller \cite{vandenDriesMiller} proved the o-minimality of the fourth example.

From now on, and for the rest of the paper, we suppose that our o-minimal structure $\mathcal{S}$ contains the semialgebraic sets. 
Recall that a set $A$ is definable if it lies in some $\mathcal{S}_n$. 
For a set $Z\subseteq \IR^{m+n}$ we call $Z_T=\{x\in \IR^n: (T,x)\in Z\}$ a fiber of $Z$.
From this viewpoint it is natural to call $Z$ a family. In particular, we call $Z$ a definable family if $Z$ is a definable set.
We write 
$\lambda_i=\lambda_i(\Lambda)$ for $i=1,\ldots,n$ for the successive minima of $\Lambda$ with respect to the zero-centered unit ball $B_0(1)$, i.e., for $i=1,...,n$
\begin{alignat*}1
\lambda_i=\inf\{\lambda:B_0(\lambda)\cap\Lambda \text{ contains $i$ linearly independent vectors}\}.
\end{alignat*}
Also recall that 
$V_j(Z_T)$ is the sum of the $j$-dimensional volumes of the orthogonal projections of $Z_T$ 
on every $j$-dimensional coordinate subspace of $\R^n$. We shall see that if $Z$ is a definable family with bounded fibers $Z_T$ then the $j$-dimensional volumes of 
the orthogonal projections of $Z_T$ on any $j$-dimensional coordinate subspace of $\R^n$ exist and are finite, and also the volume $\Vol(Z_T)$ exists and is finite.
\begin{theorem}\label{maintheorem}
Let $m$ and $n$ be positive integers, let $Z\subseteq \IR^{m+n}$ be a definable family, and suppose the  fibers $Z_T$ are bounded. 
Then there exists a constant $c_Z \in \R$, depending only on the family $Z$, such that
\begin{alignat*}1
\left||Z_T\cap \Lambda|-\frac{\Vol(Z_T)}{\det \Lambda}\right|\leq c_{Z}\sum_{j=0}^{n-1}\frac{V_j(Z_T)}{\lambda_1\cdots \lambda_j},
\end{alignat*}
where for $j=0$ the term in the sum is to be understood as $1$.
\end{theorem}
Up to the constant $c_Z$, our estimate is best-possible.  To see this we take $\Lambda=\lambda_1e_1\IZ+\cdots+\lambda_ne_n\IZ$ with $0<\lambda_1 \leq \cdots \leq \lambda_n$, 
and the semialgebraic set $Z$, defined as the union of
$Z^{(j)}=\{(T,x)\in \R^{1+n}:T\geq 0, x\in ([0,T]^{j}\times\{0\}^{n-j}+\lambda_je_j)\}$
taken over $j=1,\ldots,n-1>0$. Hence, for $T\geq 0$ we get 
\begin{alignat*}1
\left||Z_T\cap \Lambda|-\frac{\Vol(Z_T)}{\det \Lambda}\right|=\sum_{j=1}^{n-1}\prod_{p=1}^{j}\left(\left[\frac{T}{\lambda_p}\right]+1\right)
\geq 2^{-n}\sum_{j=0}^{n-1}\frac{V_j(Z_T)}{\lambda_1\cdots\lambda_j}.
\end{alignat*}
Next let us consider a simple application. Suppose we want to count lattice points in the fibers $Z_T$
of the family $Z$ as defined in (\ref{Zfunc})
by the $2^{n}$ polynomial functions $f_I(T,x)=\prod_I x_i^2 -T^2$, where $I$ runs over all subsets of $\{1,2,\ldots,n\}$, $n\geq 2$.
This problem occurs if one counts algebraic integers in a totally real field $k$, and of bounded Weil height.
Now we have $\Vol(Z_T)=2^{n}T(\log T)^{n-1}+O(T(\log T)^{n-2})$,
and moreover, $V_j(Z_T)=O(T(\log T)^{n-2})$. Obviously, our family $Z$ is a semialgebraic set.
Applying Theorem \ref{maintheorem} we get an asymptotic formula. 

Now suppose we want to derive a similar statement from the counting results in \cite{1} or \cite{art1} (\cite{3} cannot be applied as $Z_T$ is not homogeneously expanding). 
Then we require to parameterize  the boundary of $Z_T$ 
by a finite number of Lipschitz maps
$\phi:[0,1]^{n-1}\rightarrow \R^{n}$. This can certainly be done, even with a single map. But the diameter of $Z_T$ has size of order $T$, and thus
the Lipschitz constant $L$ of this map is necessarily of this size. This gives an error term of order $T^{n-1}$ which exceeds the ``main term'', at least if $n>2$. 
Possibly one can resolve this problem by using many parameterizing maps instead of just one. But even in this single case it is 
not obvious how to do this. 

Now the aforementioned example of counting integers in $k$ of bounded height is covered by more general and precise results in \cite{Widmerintpts}.
But in a subsequent paper \cite{Barroeroint} the first author will apply Theorem \ref{maintheorem} to deduce the asymptotics of algebraic integers of bounded height and of 
fixed degree over a given number field $k$. The special case $k=\Q$ follows from a result of Chern and Vaaler \cite{43} but the general result appears to be new. 

In an ongoing project we give a more elaborate application of Theorem \ref{maintheorem}, which, in conjunction with previous results of the second author, 
might lead to some new instances of Manin's conjecture
on the number of $k$-rational points of bounded height on the symmetric square of $\IP^n$, where $k$ is an arbitrary number field.
The special case $k=\IQ$ follows easily from a theorem of Schmidt \cite[Theorem 4a]{14}, which in turn follows from his results on the number of quadratic points 
of bounded height \cite[Theorem 3a]{14} and Davenport's theorem.\\

In recent times o-minimal structures have successfully been used for problems in number theory. Using ideas that date back to a paper by Bombieri and Pila \cite{BombieriPila}, 
and were further developed in various articles of Pila, 
Pila and Wilkie \cite{Pila} gave upper bounds for the number of rational points 
of bounded height on the transcendental part of definable sets. These results in turn have been applied to 
problems in Diophantine geometry (see \cite{PilaZannier}, \cite{Pila11}, \cite{MasserZannier08}, \cite{MasserZannier10} and \cite{HabeggerPila}). 
However, to the best of the authors' knowledge, o-minimal structures have not been
used so far to establish asymptotic counting results.\\

The paper is organized as follows. In Section \ref{geomnumbers} we use Geometry of Numbers, and follow arguments of Thunder \cite{34} to generalize Davenport's theorem
to arbitrary lattices with an error term as in (\ref{Thunderbound}). In Section \ref{o-minimalstructures} we collect some basic facts about o-minimal structures,
as well as some deeper results like the cell-decomposition Theorem, the Reparametrization Lemma (originally due to Yomdin \cite{Yomdin1}, \cite{Yomdin2}, and Gromov \cite[p.232]{Gromov},
and refined by Pila and Wilkie \cite{Pila}), and the existence of definable Skolem functions. 
Then, in Section \ref{Davenportconstant}, we use the fact that
there are uniform upper bounds for the number of connected components of fibers of definable sets, to establish a uniform upper bound for our quantity $h'$.
In Section \ref{sect_geomineq} we establish a geometric inequality that allows us to substitute $V_j'(Z_T)$ of (\ref{Thunderbound}) with $V_j(Z_T)$. 

This is the core argument of the paper, and the strategy is, roughly speaking, as follows. 
For each $1\leq j\leq n-1$ and any $j$-dimensional subspace
$\Sigma$ we construct a $j$-dimensional definable subset of $Z_T$ that
projects to $\Sigma$ with maximal volume. Locally, the volume of the
projection onto $\Sigma$ can be bounded by the sum of the volumes of the
projections onto the $j$-dimensional
coordinate spaces, so globally we only have to worry about these
projections being non-injective. However, o-minimality provides a bound
for the number of pre-images for each such projection, which is uniform
in $T$ and $\Sigma$, and this is sufficient.

To carry out the aforementioned strategy we require some concepts and results from geometric measure theory such as rectifiability and Hausdorff measure/dimension, 
which we derive and recall 
in Section \ref{sect_measuretheoretic_prelim}. The Reparametrization Lemma implies the required rectifiability assumptions for bounded definable sets. Finally, 
in Section \ref{Proof of Theorem} we put all together to prove Theorem \ref{maintheorem}.


Some of the potential users of our theorem may not be familiar with o-minimality. Therefore, we have given definitions, and proofs or references, 
even for the most basic concepts, and results. 
For the same reason we also have restricted ourselves to the set-theoretic language instead of the model-theoretic approach, 
although the latter often leads to simpler and quicker proofs. 

\section{Geometry of numbers}\label{geomnumbers}
 
By \cite[Lemma 8 p.135]{18} there exists a basis ${v}_1, \ldots , {v}_n$ of the lattice $\Lambda$ such that $|{v}_i|\leq i \lambda_i$ for $i=1, \ldots , n$. 
We let $\Ppsi$ be the automorphism of $\R^n$ defined by $\Ppsi({v}_i)={e}_i$, where ${e}_1=(1,0,\ldots,0),\ldots, {e}_n=(0,\ldots,0,1)$ is the
standard basis of $\IR^n$. 
Hence, we have $\Ppsi(\Lambda)=\IZ^n$.
\begin{lemma}\label{applyDavenport}
Let $\D\subseteq \R^n$ be a compact set such that $\Ppsi(\D)$ satisfies the hypothesis (1) and (2) of Theorem \ref{thmDav}. Then
\begin{alignat*}1
\left||\D\cap \Lambda|-\frac{\Vol(\D)}{\det \Lambda}\right|\leq \sum_{j=0}^{n-1}h^{n-j}V_j(\Psi(\D)),
\end{alignat*}
\end{lemma}
\begin{proof}
Clearly, we have
$$
|\D\cap \Lambda |=|\Ppsi ( \D)\cap \Z^n|,
$$
and $\Vol(\Ppsi(\D))=|\det\Ppsi |\Vol(\D)$. The inverse of $\Ppsi$ corresponds to the matrix with columns ${v}_1, \ldots , {v}_n$, and therefore $|\det\Ppsi|^{-1}=\det \Lambda$.
As $\D$ is compact also $\Ppsi(\D)$ is compact. Applying Theorem \ref{thmDav} yields the claim.
\end{proof}

In the next two lemmas we simply reproduce arguments of Thunder from \cite{34} to obtain an error term as anticipated in (\ref{Thunderbound}).

Let $1\leq j \leq n-1$, let $I$ be any subset of $\{ 1, \ldots ,n \}$ of cardinality $j$, and let $\overline{I}$ be its complement. Let $\Sigma_I$ and $\Lambda_I$ be 
respectively the subspace 
of $\R^n$ and the sublattice of ${v}_1\IZ+\cdots +{v}_n\IZ$ generated by the vectors ${v}_i$, $i \in I$. For any set $\D \subseteq \R^n$ we define
$$
\D^I=\left\lbrace {x} \in \Sigma_I :{x}+{y} \in \D \mbox{ for some } {y} \in  \Sigma_{\overline{I}} \right\rbrace .
$$
This is nothing but the projection of $\D$ to $\Sigma_I$ with respect to $\Sigma_{\overline{I}}$. 

\begin{lemma}\label{PsiV}
Suppose $\D\subseteq \R^n$ is compact. Then, for every $j=1,\ldots, n-1$,
\begin{equation*}
V_j(\Psi(\D)) \leq\sum_{|I|=j}\frac{2^j}{B_j}\frac{\Vol_j\left(\D^I\right)    }{ \lambda_1\cdots \lambda_j},   
\end{equation*}
where $B_j$ is the volume of the $j$-dimensional unit-ball. 
\end{lemma}
\begin{proof}
The orthogonal projection of $\Psi(\D)$ to the coordinate subspace spanned by ${e}_i$, $i\in I$ for some choice 
of $I$, corresponds to the projection $\D^I$ of $\D$ to $\Sigma_I $ with respect to $\Sigma_{\overline{I}} $.
Therefore we have that
$$
V_j(\Psi(\D))=\sum_{|I|=j}\frac{\Vol_j\left(\D^I\right)    }{\mbox{det}\Lambda_I} .     
$$
As $\lambda_i(\Lambda_I)\geq \lambda_i$ for $1\leq i\leq j$ we deduce from Minkowski's second theorem
$$
\mbox{det}\Lambda_I \geq  \frac{B_j}{2^j} \lambda_1\cdots \lambda_j,
$$
and this proves the lemma.
 \end{proof}

\begin{definition}\label{defV'}
Suppose $\D\subseteq \IR^n$ is compact, and suppose $0<j<n$. We define $V'_j(\D)$ to be the supremum of the volumes of the orthogonal projections of $\D$ 
to any $j$-dimensional linear subspace of $\R^n$, and we set
$V'_0(\D)=1$.
\end{definition}
\begin{lemma}\label{VI}
Suppose $\D\subseteq \R^n$ is compact. Then
for any $j=1,\ldots ,n-1$ and any $I\subseteq \{ 1, \ldots ,n \}$ with $|I|=j$ there exists a constant $c=c(n,j)$ such that 
$$
\Vol_j\left(\D^I\right)\leq c V'_j(\D).
$$
\end{lemma}
\begin{proof}

Let ${v}_i'$ be the vectors defined by
$$
{v}_i'= \frac{{v}_1\wedge \cdots \wedge {v}_{i-1} \wedge {v}_{i+1} \wedge \cdots \wedge {v}_n }{|{v}_1\wedge  \cdots \wedge {v}_n |} =\frac{{v}_1\wedge \cdots \wedge {v}_{i-1} 
\wedge {v}_{i+1} \wedge \cdots \wedge {v}_n }{\text{det}\Lambda} .
$$
Now let $\Sigma_{\overline{I}}^\bot$ be the linear subspace generated by ${v}_i'$, $i\in I$ (and thus orthogonal to $ \Sigma_{\overline{I}} $). Let $\widehat{\D^I}$ be 
the orthogonal projection of $\D$ on $\Sigma_{\overline{I}}^\bot$. This means
$$
\widehat{\D^I}=\left\lbrace {x} \in \Sigma_{\overline{I}}^\bot:{x}+{y} \in \D \mbox{ for some } 
{y} \in  \Sigma_{\overline{I}} \right\rbrace .
$$
There exists a linear transformation $\varphi$ between $\Sigma_I$ and $\Sigma_{\overline{I}}^\bot$ that maps a point of $\Sigma_I$ to its orthogonal projection on 
$\Sigma_{\overline{I}}^\bot$. Note that $\varphi(\D^I)\subseteq \widehat{\D^I}$ because, for every ${x} \in \D^I$, ${x} ={z}+{y} $ for some ${z} \in \D$ and 
${y} \in \Sigma_{\overline{I}}$, and $\varphi({x})={x}+ {y}' $ for some ${y}' \in \Sigma_{\overline{I}}$, and thus $\varphi({x})={z}+({y}+{y}') \in  \widehat{\D^I}$. 
Moreover, $\varphi$ is an injective map. Indeed, suppose we had ${x},{y} \in \Sigma_I$ with the same image, then ${x}-{y} \in \Sigma_{\overline{I}} \cap \Sigma_I$, which 
means ${x}={y}$. Therefore we can see $\varphi$ as an automorphism of $\R^j$. We want to bound the determinant of the inverse of $\varphi$. Let
$$
{x}= \sum_{i\in I} a_i {v}_i  \in \Sigma_{I}.
$$
Since $  {x} - \varphi({x})\in \Sigma_{\overline{I}}$ and by definition ${v}_p\cdot{v}_q' = \delta_{pq}$,  
we have, for every $i \in I$, $({x} - \varphi({x}))\cdot{v}_i'=0 $ and $a_i={x}\cdot {v}_i'=\varphi({x})\cdot {v}_i'$. Thus,
$$
|{x}| \leq \sum_{i\in I} |a_i| |{v}_i| \leq \sum_{i \in I}|\varphi({x})| \left| {v}_i'\right|| {v}_i| .
$$
The condition $| {v}_i|\leq i \lambda_i$, the definition of ${v}'_i$ and Minkowski's second Theorem imply that
$$
 | {v}_i'|| {v}_i|\leq \frac{\prod_p |{v}_p | }{\text{det}\Lambda} \leq \frac{n!\prod_p \lambda_p }{\text{det}\Lambda} \leq \frac{n! 2^n}{B_n}.
$$
Thus,
$$
|{x}| \leq j  \frac{n! 2^n}{B_n} |\varphi({x})|,
$$
and this implies
$$
\| \varphi^{{-1}} \|_{op} \leq  j  \frac{n! 2^n}{B_n},
$$
where $\| \cdot \|_{op}$ is the operator norm. Suppose $\varphi^{-1}$ corresponds to the matrix 
$\left( a_{pq} \right)_{p,q=1}^j$ then 
$\| \varphi^{{-1}} \|_{op} \geq \max_{p,q}\left\lbrace \left| a_{pq}\right| \right\rbrace $. By Hadamard's inequality
$$
\left|\det\left( \varphi^{{-1}}\right)\right| \leq \prod_{p=1}^{j}\left( \sum_{q=1}^j   a_{pq}^2 \right)^{1/2}\leq \left(  \sqrt{j} \| \varphi^{{-1}} \|_{op} \right)^j.
$$
Finally, since $\D^I\subseteq \varphi^{-1}\left(\widehat{\D^I}\right)$,
$$
\Vol_{j}\left(\D^I\right)\leq \Vol_j\left(\varphi^{-1}\left(\widehat{\D^I}\right)\right)\leq \left( j^{3/2}  \frac{n! 2^n}{B_n}  \right)^j  
\Vol_{j}\left(\widehat{\D^I}\right)\leq  \left( j^{3/2}  \frac{n! 2^n}{B_n}  \right)^j V'_j(\D) .
$$

\end{proof}

\section{O-minimal structures}\label{o-minimalstructures}
In this section we state the basic properties used later on. Most of the results are taken literally from \cite{vandenDries}.

We start with a list of simple facts that will be used in the sequel, sometimes without explicitly referring to them.
\begin{lemma}\label{basicfactsdefsets}
\hfill 
\begin{enumerate}[i)]
\item $A, B \in \mathcal{S}_n \Rightarrow A\cap B\in \mathcal{S}_n$;
\item $A\in \mathcal{S}_n, B\in \mathcal{S}_m \Rightarrow A\times B\in \mathcal{S}_{n+m}$;
\item $A\in \mathcal{S}_n, 1\leq k\leq n \Rightarrow \{(x_1,\ldots,x_k,x_1,\ldots,x_n): (x_1,\ldots,x_n)\in A\}\in \mathcal{S}_{k+n}$;
\item $A\in \mathcal{S}_n$, $\sigma$ a permutation on $n$ coordinates $\Rightarrow \sigma A\in \mathcal{S}_n$;
\item $A\in \mathcal{S}_n \Rightarrow \pi_C(A) \in \mathcal{S}_n$, where $C$ is a coordinate subspace in $\R^n$ and $\pi_C$ 
is the orthogonal projection to $C$;
\item $S\in \mathcal{S}_{m+n}, a\in \R^m \Rightarrow S_a=\{x\in \R^n: (a,x)\in S\} \in \mathcal{S}_n$.
\end{enumerate}
\end{lemma}

\begin{proof}
The statement $i)$ is obvious from Definition \ref{defomin}. 
For $ii)$ we use that $A\times B=A\times \R^m\cap \R^n\times B$.
Now $iii)$ follows easily.
For $iv)$ we note that $\sigma A$ is the projection
to the first $n$ coordinates of the definable set $\cap_{i=1}^{n}\{(u,x)\in \R^n\times A: u_i=x_{\sigma(i)}\}$.
Then, $v)$ follows immediately. Finally, for $vi)$ we note that $S_a=\pi(S\cap \{a\}\times \R^n)$, where $\pi$
projects to the last $n$ coordinates. 
\end{proof}

Recall that a subset $X$ of $\IR^n$ is definable (in the o-minimal structure $\mathcal{S}$) if $X\in \mathcal{S}_n$.
Also recall that our o-minimal structure $\mathcal{S}$ contains the semialgebraic sets. 

\begin{definition} \label{defdefsets}
Suppose $X\subseteq  \IR^n$ is definable then we say that $f:X\rightarrow \IR^m$ is a definable function (in $\mathcal{S}$) 
if its graph $\Gamma(f)=\{(x,f(x)): x\in X\}$ is definable (in $\mathcal{S}$). 
We say that $f$ is bounded if its graph is a bounded set. 
\end{definition}
Let $\varphi$ be an endomorphism of $\R^n$. 
Then we will identify $\varphi$ with the vector $(\varphi(e_1),\ldots,\varphi(e_n))\in \R^{n^2}$, where $e_1,\ldots,e_n$ is the standard basis of $\IR^n$. 
A set of the form
\begin{equation} \label{exalgset}
\left\lbrace (\varphi, x, y)\in \R^{n^2+2n}: y=\varphi(x)  \right\rbrace, 
\end{equation}
is defined by polynomial equalities, and hence is definable.

Now suppose $X$ is a definable set, and let 
$$
C(X)=\{f:X\rightarrow \IR: f \text{ is definable and continuous}\},
$$
and
$$
C_\infty(X)=C(X)\cup\{-\infty, \infty\}.
$$
For $f$ and $g$ in $C_\infty(X)$ we write $f<g$ if $f(x)<g(x)$ for all $x\in X$. In this case we put
\begin{alignat*}1
(f,g)_X=\{(x,r)\in X\times \IR: f(x)<r<g(x)\}.
\end{alignat*}
It is not difficult to see that $(f,g)_X$ is a definable subset of $\IR^{n+1}$, e.g., $(-\infty,g)_X$ is a projection
of the definable set $\{(x,z,y,z)\in \Gamma(g)\times \R^2: y<z\}$.\\

We now come to the definition of cells which are particularly simple definable sets.
\begin{definition}\label{Defcell}
Let $(i_1,\ldots, i_n)$ be a sequence of zeros and ones of length $n$. A $(i_1,\ldots, i_n)$-cell is a 
definable subset of $\IR^n$ obtained by induction on $n$ as follows:
\begin{enumerate}
\item A $(0)$-cell is a one-element set $\{r\}\subseteq \IR$, a $(1)$-cell is a nonempty interval $(a,b)\subseteq \IR$.
\item Suppose $(i_1,\ldots, i_n)$-cells are already defined; then a $(i_1,\ldots, i_n,0)$-cell is the graph $\Gamma(f)$ of a function $f\in C(X)$, 
where $X$ is a $(i_1,\ldots, i_n)$-cell;
further, a $(i_1,\ldots, i_n,1)$-cell is a set $(f,g)_X$, where $X$ is a $(i_1,\ldots, i_n)$-cell and $f,g \in C_\infty(X)$ with $f<g$.
\end{enumerate}
A cell in $\IR^n$ is an $(i_1,\ldots, i_n)$-cell for some (necessarily unique) sequence $(i_1,\ldots, i_n)$.\\
\end{definition}
\begin{lemma} \label{conlem}
Each cell is connected in the usual topological sense. 
\end{lemma}

\begin{proof}
This follows from \cite[Exercise 7, p.59]{vandenDries} combined with \cite[Ch.3, (2.9) Proposition]{vandenDries}. 
\end{proof}

We need another definition.
\begin{definition}
A decomposition of $\IR^n$ is a special kind of partition into finitely many cells. 
Again the definition is by induction on $n$:\\
(1) a decomposition of  $\IR$ is a collection 
\begin{alignat*}1
\{(-\infty,a_1), (a_1,a_2),\ldots,(a_k,\infty),\{a_1\},\ldots,\{a_k\}\},
\end{alignat*}
where $a_1<\cdots <a_k$ are points in $\IR$.\\
(2) a decomposition of $\IR^{n+1}$ is a finite partition of $\IR^{n+1}$ into cells $A$ such that the set of projections $\pi(A)$
is a decomposition of $\IR^n$. (Here $\pi: \IR^{n+1}\rightarrow \IR^n$ is the usual projection map on the first $n$ coordinates.)
\end{definition}

A decomposition $\mathcal{D}$ of $\IR^n$ is said to partition a set $S \subseteq \IR^n$ if each cell in $\mathcal{D}$ is either part of $S$ or disjoint from $S$. 
We can now state the following theorem, which is a special case of the cell decomposition theorem (\cite[Ch.3, (2.11)]{vandenDries} or \cite[4.2]{vandenDries96}).

\begin{theorem} \label{celldecthm}

Given a definable set $S \subseteq \IR^n$ there is a decomposition of $\IR^n$ partitioning $S$.

\end{theorem}
\begin{proof}
This follows immediately from  $(I_n)$ in \cite[Ch.3, (2.11)]{vandenDries}.
\end{proof}

We recall the definition of dimension of a definable set from \cite[Ch.4]{vandenDries}.
\begin{definition}\label{dimension}
Let  $S\subseteq \R^n$ be nonempty and definable. The dimension of $S$ is defined as
$$
\dim S = \max \{i_1+\cdots +i_n : S \mbox{ contains an } (i_1, \ldots , i_n)-\mbox{cell}  \}.
$$
To the empty set we assign the dimension $-\infty$.
\end{definition}
Note that a definable set of dimension zero is a finite collection of points.
Next we collect some basic facts about definable functions. These will be used in the sequel, sometimes without further mention.
\begin{lemma}\label{Propdefdim}
Suppose $f:A\rightarrow B$ is a definable function and suppose $C$ is a nonempty definable subset of $A$. Then
\begin{enumerate}[i)]
\item $A$ and $f(A)$ are definable;
\item The restriction $f\mid_C:C\rightarrow B$ is definable;
\item If $f$ is bijective then $f^{-1}:B\rightarrow A$ is definable;
\item If $f$ is bijective then $\dim A=\dim B$. 
\end{enumerate}
\end{lemma}
\begin{proof}
The claim $i)$ follows immediately from the definition, similarly $ii)$ by noting that $\Gamma(f\mid_C)=\Gamma(f)\cap \left( C\times f(A) \right)$, and $iii)$ is obvious.
For $iv)$ we refer to \cite[Ch.4, (1.3) Proposition (ii)]{vandenDries},
\end{proof}

\begin{definition}\label{p-parametrization}
Let $S \subseteq \R^n$ be a definable set of dimension $d>0$. 
Let $\P$ be a finite set of definable functions $\phi: (0,1)^d \rightarrow S$ such that 
$\bigcup_{\phi \in \P} \phi\left(( 0,1)^d\right)=S$. We call $\P$ a parametrization of $S$. Let $\alpha \in (\N\cup\{0\})^d$ be a multi index write $|\alpha|=\sum \alpha_i$ and, 
for $\phi=(\phi_1, \dots , \phi_n) \in \P$,
$$
\phi^{(\alpha)}= \left( \frac{\partial^{|\alpha|}\phi_1}{\partial^{\alpha_1} x_1 \cdots \partial^{\alpha_d} x_d  }  , \cdots , 
\frac{\partial^{|\alpha|}\phi_n}{\partial^{\alpha_1} x_1 \cdots \partial^{\alpha_d} x_d  } \right).
$$ 
We call $\P$ a $p$-parametrization if every $\phi \in \P$ is of class $C^{(p)}$ and has the property that $\phi^{(\alpha)}$ is bounded for each
$\alpha \in (\N\cup\{0\})^d$ with $|\alpha|\leq p$.

\end{definition}

\begin{theorem}[Pila, Wilkie]\label{PWpar}
For any $p\in \N$, and any bounded definable set $S$ of positive dimension, there exists a $p$-parametrization of $S$.
\end{theorem}
\begin{proof}
This is a special case of \cite[Theorem 2.3]{Pila}. 
\end{proof}

Let $D\subseteq \R^n$ be nonempty. We say $f: D \rightarrow \R^m$ is a Lipschitz map if there exists a real constant $L$ such that 
$$
|f(x)-f(y)|\leq L|x-y| \text{ for all } x,y \in D.
$$
\begin{corollary}\label{lippar}
Let $S\subseteq \R^n$ be bounded and definable, and suppose $\dim S=d>0$.
Then $S$ can be parameterized by a finite number of Lipschitz maps $\phi: (0,1)^d \rightarrow S$.
\end{corollary}

\begin{proof}

By Theorem \ref{PWpar} any bounded definable set $S$ of dimension $d$ can be parameterized by a finite number of maps $\phi: (0,1)^d \rightarrow S$ with
uniformly bounded partial derivatives. 
This  implies the claim  (see also \cite[Ch.7, (2.8) Lemma]{vandenDries}).

\end{proof}

\begin{proposition}\cite[Ch.3, (3.5) Proposition]{vandenDries}\label{PropDecomposition}
Let $\pi:\IR^{m+n} \rightarrow \IR^m$ be the projection on the first $m$ coordinates. If $C$ is a cell in $\IR^{m+n} $ and $a \in \pi(C)$, then $C_a $ is a cell in $\IR^n$.
Moreover, if $\mathcal{D}$ is a decomposition of $\IR^{m+n}$ and $a \in \IR^m$ then the collection
$$
\mathcal{D}_a:=\{ C_a: C \in \mathcal{D}, a \in \pi(C) \}
$$
is a decomposition of $\IR^n$.
\end{proposition}

\begin{corollary}\label{corollconncomp}

Let $S \subseteq \IR^{m+n}$ be a definable family. Then there exists a number $M_S \in \IN$ such that for each $a \in \IR^m$ the set $S_a \subseteq \IR^n$ can be partitioned into at 
most $M_S$ cells. In particular, each fiber $S_a$ has at most $M_S$ connected components.

\end{corollary}

\begin{proof}
By the cell decomposition theorem there exists a decomposition $\mathcal{D}$ of $\IR^{m+n}$ partitioning $S$. Then for each $a \in \IR^m$ the decomposition $\mathcal{D}_a$ 
of $\IR^n$ consists of at most $|\mathcal{D}|$ cells and partitions $S_a$. So we can take $M_S=|\mathcal{D}|$. The last statement follows from Lemma \ref{conlem}.
\end{proof}
Another important property of o-minimal structures is the possibility of ``lifting'' projections. In model-theoretic terms this might be rephrased as
existence of definable Skolem functions.

\begin{proposition}\cite[Ch.6, (1.2) Proposition]{vandenDries}\label{skolem}
If $S\subseteq \R^{m+n}$ is definable and $\pi: \R^{m+n} \rightarrow \R^m$ is the projection on the first $m$ coordinates, then there is a definable map $f:\pi(S)\rightarrow \R^n$ 
such that $\Gamma(f)\subseteq S$. 
\end{proposition}

The proof of \cite[Ch.6, (1.2) Proposition]{vandenDries} actually shows that there is an algorithmic way to construct the Skolem function $f$.
The construction of $f$ is of no importance for us but we will use the fact that this choice of $f$ is determined by $S$ and $\pi$.

We write $\cl(A)$ and $\inte(A)$ for the the topological closure and the interior of the set $A$ respectively.  
Also recall that $\bd(A)$ denotes the topological boundary of $A$.

\begin{lemma}\label{collintfibers}
Suppose $Z\subseteq \R^{m+n}$ is definable. Then 
$\{(T,x): x\in \inte(Z_T)\}$, $\{(T,x): x\in \cl(Z_T)\}$, and $\{(T,x): x\in \bd(Z_T)\}$ are definable.  
\end{lemma}
\begin{proof}
The first statement is \cite[Ch.1, (3.7) Exercise (ii)]{vandenDries}. For the second set note that $x\in \cl(Z_T)$ is equivalent to 
$x\notin \inte(\R^n\backslash Z_T)$, and, moreover, $\R^n\backslash Z_T=(\R^{m+n}\backslash Z)_T$. Hence, 
$\{(T,x): x\in \cl(Z_T)\}=\R^{m+n}\backslash \{(T,x): x\in \inte((\R^{m+n}\backslash Z)_T)\}$, which is definable by our first statement.
Finally, as $\{(T,x): x\in \bd(Z_T)\}=\{(T,x): x\in \cl(Z_T)\}\backslash\{(T,x): x\in \inte(Z_T)\}$ we get the last statement.
\end{proof}

\section{The Davenport constant}\label{Davenportconstant}
If $D\subseteq \R^n$ satisfies the conditions (1) and (2) in Theorem \ref{thmDav} then we say $h$ is a Davenport constant for $D$.
Of course, this has nothing to do with the classical Davenport constant of a finite abelian group.
\begin{lemma} \label{lemdav}

Let $Z \subseteq \IR^{m+n}$ be a definable family. There exists a natural number $M=M_Z$, depending only on $Z$, such that  
for every $T \in \IR^m$ and every endomorphism $\Ppsi$ of $\IR^n$ the number $M$ is a Davenport constant for $\Ppsi(Z_T)$.

\end{lemma}

\begin{proof}

Let $I$ be a nonempty subset of $\{1, \ldots ,n\}$ and let $\pi_{C_I}$ be the orthogonal projection of $\IR^n$ on the coordinate subspace $C_I$ generated by the ${e}_i$, $i \in I$. 
Recall the notation of (\ref{exalgset}) in Section \ref{o-minimalstructures} and let $W$ be the set
\begin{equation}\label{defW}
  W=\left\{  (\Ppsi,T,x) \in \IR^{n^2+m+n}:   x \in \Ppsi(Z_T) \right\}.
\end{equation}
Note that, up to a coordinate permutation, $W$ is the projection to the first $n^2+m+n$ coordinates of the definable set $\left\lbrace \left( \Ppsi,x,T,y\right) \in \IR^{n^2+n+m+n}:   x=\Ppsi(y)\right\rbrace \cap \left(\R^{n^2+n}\times Z \right)$. By Lemma \ref{basicfactsdefsets} and the fact that semialgebraic sets are definable, this is a definable set. Moreover, note that
$$
W_{(\Ppsi,T)}=\Ppsi(Z_T).
$$

Let us set some notation we need. We indicate by $\pi'_{C_I}$ the endomorphism of $\R^{n^2+m+n}$ defined by $(\Ppsi ,T,x)\mapsto (\Ppsi, T, \pi_{C_I}(x))$. 
A line in $C_I$ parallel to ${e}_{i_0}$ is determined by $|I|-1$ reals and therefore we indicate it by $(l_i)_{i\in I\setminus \{i_0 \}}$.

Let $I \subseteq \{ 1,\ldots ,n \}$ be nonempty and $i_0 \in I$, we consider the sets
\begin{align*}
B^{I,(i_0)}= \left\{  \left((l_i)_{i\in I\setminus \{i_0 \}} ,\Ppsi,T,x \right) \in \IR^{|I|-1} \times\IR^{n^2+m+n}: \right. \\
\left. \vphantom{ \left((l_i)_{i\in I\setminus \{i_0 \}} ,\Ppsi,T,x \right) \in \IR^{|I|-1} \times\IR^{n^2+m+n}:} (\Ppsi,T,x) \in \pi'_{C_I}(W), l_i=x_i \text{ for }  i\in I\setminus \{i_0 \} \right\}.
\end{align*}
Again by elementary properties mentioned in Section \ref{o-minimalstructures}, these are definable sets . A fiber $B^{I,(i_0)}_{\left((l_i) ,\Ppsi,T\right)}$ is exactly the intersection 
of $\pi'_{C_I}(W)_{(\Ppsi,T)}=\pi_{C_I}(W_{(\Ppsi,T)})=  \pi_{C_I}(\Ppsi(Z_{T}))$ and the line $(l_i)_{i\in I\setminus \{i_0 \}} $ parallel to ${e}_{i_0}$ in the subspace $C_I$.

Now we use Corollary \ref{corollconncomp} to find a uniform bound $M^{I,(i_0)}$ for the number of connected components of the fibers $B^{I,(i_0)}_{\left((l_i) ,\Ppsi,T\right)}$ of $B^{I,(i_0)}$. This means that $M^{I,(i_0)}$ is a bound on the number of connected components of the intersection of $\pi_{C_I}(\Ppsi(Z_{T}))$ with any line of $C_I$ parallel to $e_{i_0}$, for any choice of $\Ppsi$ and $T$.
Finally, we can take $M$ to be the maximum of the $M^{I,(i_0)}$ for all the possible choices of $I$ and $i_0 \in I$.

\end{proof}


\section{Hausdorff measure and rectifiability}\label{sect_measuretheoretic_prelim}

We also require the $j$-Hausdorff measure $\H^j$. For the definition and properties of the Hausdorff measure we refer to \cite{Federer} or \cite{Morgan}.

\begin{lemma}\label{lemhaus}
Suppose $1\leq j\leq n$, $A\subseteq \R^n$ and suppose $A$ is $j$-Hausdorff measurable.
Furthermore, let $\varphi:\R^n\rightarrow \R^n$ be an endomorphism.  
Then $\H^j(\varphi(A))\leq \|\varphi\|_{op}^j\H^j(A)$. Moreover, if $\varphi$ is an orthogonal projection we have $\H^j(\varphi(A))\leq \H^j(A)$. If $\varphi$ is 
in the orthogonal group $O_n(\R)$ then we have $\H^j(\varphi(A))=\H^j(A)$. 
\end{lemma}
\begin{proof}
The first claim follows from \cite[2.4.1 Theorem 1]{Evans}. If $\varphi$ is in $O_n(\R)$ or if $\varphi$ is an orthogonal projection then  $\|\varphi\|_{op}=1$. If 
$\varphi\in O_n(\R)$ then also 
$\varphi^{-1}\in O_n(\R)$, and we apply the previous with $\varphi^{-1}$ and $\varphi(A)$.
\end{proof}

\begin{proposition}\label{defdimhausdorffdim}
Suppose $A \subseteq \R^n$ is nonempty and definable. Then $\dim A$ coincides with the Hausdorff dimension. Moreover,  
if $\dim A=d$ and $A$ is bounded, then $A$ is $j$-Hausdorff measurable for every $j$ with $ d\leq j\leq n$. Finally, $\H^d(A)< \infty$ and $\H^j(A)=0$ for $j>d$. 
\end{proposition}
\begin{proof}
See \cite[last paragraph on p.177]{vandenDriesLimitsets}. The last claim follows from the definition of Hausdorff dimension.
\end{proof}

It is well known that on $\R^n$ the $n$-Hausdorff measure coincides with the Lebesgue measure (see \cite[2.8. Corollary]{Morgan}). This, 
together with Proposition \ref{defdimhausdorffdim}, implies that a definable set in $\R^n$ of dimension $<n$ has volume zero. 
Also recall that any bounded set that is open or closed is measurable and has finite volume. 

\begin{lemma}\label{measdef}
Let $A\subseteq \R^n$ be a bounded definable set. Then $\Vol(\bd(A))=0$. In particular, $A$ is measurable and $\Vol(\inte(A))=\Vol(A)=\Vol(\cl(A))$.
\end{lemma}
\begin{proof}
By \cite[Ch.4, (1.10) Corollary]{vandenDries} we have $\dim \bd(A)<n$. This, combined with the previous observation yields $\Vol(\bd(A))=0$. 
\end{proof}
Berarducci and Otero \cite{Berarducci2004} have proven measurability results for more general o-minimal structures expanding a field, not necessarily $\R$. 
E.g., \cite[2.5 Theorem]{Berarducci2004} implies that any bounded definable set is measurable.

\begin{lemma}\label{Vjclos}
Let $Z\subseteq \IR^{m+n}$ be a definable family and suppose the  fibers $Z_T$ are bounded. Then for $1\leq j\leq n-1$ 
the $j$-dimensional volumes of the orthogonal projections of $Z_T$ 
on every $j$-dimensional coordinate subspace of $\R^n$ exist and are finite.
Moreover, we have
$V_j(Z_T)=V_j(\cl(Z_T))$.
\end{lemma}
\begin{proof}
Let $C$ be a coordinate space of dimension $j$, and let $\pi_C$ be the orthogonal projection from $\IR^n$ to $C$.
Recall that the Lebesgue measure on $C$ is denoted by $\Vol_j$. 
Using the continuity of $\pi_C$ we get  $\pi_C(\cl(Z_T))=\cl(\pi_C(Z_T))$. In particular, $\pi_C(\cl(Z_T))$ is measurable, 
and $\Vol_j(\pi_C(\cl(Z_T)))=\Vol_j(\cl(\pi_C(Z_T)))$.
Next we apply Lemma \ref{measdef} with $A=\pi_C(Z_T)$ in the coordinate space $C$ to get $\Vol_j(\cl(\pi_C(Z_T)))=\Vol_j(\pi_C(Z_T))$,
and this proves the claim.
\end{proof}


Next we recall the definition of $j$-rectifiability from \cite[Ch.3, 3.2.14]{Federer}.
\begin{definition}\label{krectifiable}
Let  $A \subseteq \R^n$ and let $j$ be a positive integer. 
We say $A$ is $j$-rectifiable if there exists a Lipschitz function mapping some bounded subset of $\R^j$ onto $A$. Moreover, $A$ is $(\H^j,j)$-rectifiable if there exist countably 
many $j$-rectifiable sets whose union is $\H^j$-almost $A$ and $\H^j(A)<\infty$.
\end{definition}

\begin{proposition}\label{proprect}
Let $A\subseteq \R^n$ be bounded and definable, and suppose $\dim A=d>0$.
Then $A$ is $(\H^j,j)$-rectifiable for every $j$ such that $d \leq j \leq n$.
\end{proposition}
\begin{proof}
By Corollary \ref{lippar} we can cover $A$ by the images of finitely many Lipschitz maps $\phi: (0,1)^d \rightarrow \R^n$ 
whose domain can clearly be extended to $(0,1)^j$ for every $j=d+1, \ldots ,n$ without loosing the Lipschitz condition. 
The finiteness of $\H^j(A)$ comes from Proposition \ref{defdimhausdorffdim}.
\end{proof}

We fix an integer $j\in\{1,\ldots ,n-1\}$. Let $I$ be a subset of $\{1, \ldots , n \}$ of cardinality $j$ and let $\pi_{I}:\R^n \rightarrow \R^j$ 
be the projection map such that $\pi_{I}(x_1, \ldots , x_n)= (x_i)_{i \in I}$. For $y \in \R^j$ let
\begin{equation}\label{defN}
 N(\pi_I \mid A, y) = \lvert  \{ x\in A: \pi_I(x)=y \}    \rvert = \lvert   \pi_I^{-1}(y)\cap A    \rvert .
\end{equation}
A priori, $N(\pi_I \mid A, y)$ could be infinite, even for every $y \in \pi_I(A)$. The following theorem (\cite[3.2.27 Theorem]{Federer}) tells us that if 
$A$ is $(\H^j,j)$-rectifiable then we can integrate $N(\pi_I \mid A, y)$ 
and obtain a finite value.
Unless specified otherwise, the domain of integration is always $\R^j$.

\begin{theorem}\cite[3.2.27 Theorem]{Federer}\label{thmfederer}
If $1\leq j\leq n$, and if $A$ is a $(\H^j,j)$-rectifiable subset of $\R^n$,  then
$$
\left( \sum_{|I|=j} a_I(A)^2  \right)^{\frac{1}{2}}\leq \H^j(A)\leq \sum_{|I|=j}a_I(A) ,
$$
where
$$
a_I(A)=\int N(\pi_I \mid A, y) \de \Leb^j y. 
$$
\end{theorem}
To conclude this section we apply Theorem \ref{thmfederer} to fibers of definable families.
\begin{lemma}\label{lem1}
Let $S \subseteq \R^{\p+n}$ be a definable family whose fibers $S_\a \subseteq \R^n$ are bounded and of dimension at most $j \geq 1$. 
Then there exists a real constant $\M_I=\M_I(S)$ such that
$$
\H^j(S_\a) \leq \sum_{|I|=j}\M_I \Vol_j\left(  \pi_I \left(S_{\a}\right) \right),
$$
for every $\a \in \R^\p$.
\end{lemma}

\begin{proof}

If $S = \emptyset$, the claim is trivially true. For those $a$ such that $S_\a = \emptyset$ or $\dim S_\a=0$
we have from Proposition \ref{defdimhausdorffdim} that $\H^j\left(S_\a\right)=0$, and so in this case again the claim is trivially true.
Therefore, we can assume that $\dim S_\a>0$, and so we get from Proposition \ref{proprect} that $S_\a$ is $(\H^j,j)$-rectifiable.
Hence, we can apply Theorem \ref{thmfederer}, and we get
\begin{equation*}
 \H^j\left(S_\a\right)\leq \sum_{|I|=j}\int N\left(\pi_I \mid S_\a, y\right) \de \Leb^j y,
\end{equation*}
for every $\a \in \R^\p$ such that $\dim S_\a>0$. Therefore, we are left to prove that for any $I\subseteq \{1,\ldots ,n\}$ of cardinality 
$j$ there exists a real $\M_I=\M_I(S)$ such that 
\begin{equation}\label{eq20}
\int N\left(\pi_I \mid S_\a, y\right) \de \Leb^j y \leq \M_I \Vol_j\left(  \pi_I \left(S_\a\right) \right),
\end{equation}
for every $\a\in \R^{\p}$.

Let $R$ be the definable family
$$
R=\left\lbrace (\a,y,x)\in \R^{\p+j+n}: (\a,x) \in S, y=\pi_I(x)  \right\rbrace.
$$
Note that $R_{(\a,y)}=\pi_I^{-1}(y) \cap S_\a$.
Thus, for every $(\a,y) \in \R^{\p+j}$ we have $ N\left(\pi_I \mid S_\a, y\right)=|R_{(\a,y)}|$. 
Moreover, by Corollary \ref{corollconncomp} there is a uniform upper bound $\M_I$ for the number of connected components of the fibers $R_{(\a,y)}$. 
In particular, if $\dim R_{(\a,y)}=0$ we get $ |R_{(\a,y)}|\leq \M_I$.

Now fix an $\a\in \R^\p$. The restriction ${\pi_I}_{\mid S_\a}:S_\a\rightarrow \R^j$ is a definable map. Thus, by 
\cite[Ch. 4, (1.6) Corollary (ii)]{vandenDries},
we obtain 
$$
P=\left\lbrace y \in \R^{j}: \dim \left( \pi_I^{-1}(y)\cap S_\a\right)\geq 1  \right\rbrace
$$
is definable, and, moreover,
$$
\dim P\leq \dim S_\a -1\leq j-1.
$$
Hence $P$ has measure zero in $\R^j$.
Let $Q$ be its complement in $\pi_I(S_\a)$, i.e., 
$Q=\pi_I(S_\a) \setminus P=\left\lbrace y \in \pi_I(S_\a): \dim \left(\pi_I^{-1}(y)\cap S_\a \right)=0  \right\rbrace$. 
This set is definable, and it is exactly the set of $y$ such that $R_{(\a,y)}$ has dimension zero. Therefore
$$
\int N\left(\pi_I \mid S_\a, y\right) \de \Leb^j y =\int_Q |R_{(\a,y)}| \de \Leb^j y
\leq \int_Q \M_I \de \Leb^j y=\M_I\Vol_j\left(  \pi_I \left(S_\a\right) \right).
$$

\end{proof}

\section{A geometric inequality}\label{sect_geomineq}

In this section we are going to prove the following proposition. Recall the definition of $V_j'(\cdot)$ from Definition \ref{defV'}, and also 
that $\cl(Z_T)$ denotes the topological closure of $Z_T$.

\begin{proposition}\label{geomineq}
 
Let $Z\subseteq \R^{m+n}$ be a definable family such that the fibers $Z_T$ are bounded, and let $j$ be an integer such that $0\leq j\leq n-1$. Then there exists a constant $\L_Z$, 
depending only on the family and on $j$, such that 
$$
V_j'(\cl(Z_T))\leq \L_Z V_j(Z_T),
$$
for every $T \in \R^m$. 

\end{proposition}

If $Z=\emptyset$ or $j=0$ the inequality is trivially true. For the remainder of this section we assume that $Z$ is nonempty, and we fix an integer $j$ satisfying $1\leq j\leq n-1$.
By Lemma \ref{Vjclos} we have $V_j(Z_T)=V_j(\cl(Z_T))$. Hence, for the rest of this section we can and will also assume $$\cl(Z_T)=Z_T.$$ 
Let $O_n(\R)$ be the orthogonal group. It embeds into $\R^{n^2}$ if we identify, as already done before, a linear function $\varphi$ with the image vector of the standard basis. 
So $O_n(\R)$ is a semialgebraic set, as it is defined by polynomial equalities.

\begin{lemma}\label{lemfam}
There exists a definable set $Z'\subseteq \R^{n^2+m+n}$ depending only on $Z$ such that 
\begin{equation}\label{ineqdim}
\dim Z'_{(\varphi,T)} \leq j ,
\end{equation}
and
\begin{equation}\label{subset}
Z'_{(\varphi,T)}\subseteq Z_T,
\end{equation}
for every $(\varphi,T) \in \R^{n^2+m}$, and
\begin{equation}\label{ineqhaus}
V'_j(Z_T)\leq \sup_{\varphi \in O_n(\R)} \H^j\left(Z'_{(\varphi,T)}\right),
\end{equation}
for every $T\in \R^m$.
\end{lemma}

\begin{proof}
 
Let
$$
S =\{  (\varphi, T, y ) \in  \R^{n^2+m+n} : \varphi  \in O_n(\R),  y \in \varphi(Z_T) \}.
$$
This set is nothing but the set $W$ in (\ref{defW}) intersected with $O_n(\R)\times \R^{m+n}$ and is therefore definable. Note that
\begin{equation} \label{eq1}
S_{(\varphi,T)}=\varphi(Z_T),
\end{equation}
for every $(\varphi,T) \in O_n(\R)\times \R^{m}$.
Let $\pi:\R^{n^2+m+n} \rightarrow \R^{n^2+m+j}$ be the projection that cancels the last $n-j$ coordinates. We use the fact that o-minimal structures have definable Skolem functions 
(Proposition \ref{skolem}, see also the observation after Proposition \ref{skolem}). There exists an explicit construction of a definable function 
$$
f:\pi(S)\subseteq \R^{n^2+m+j} \rightarrow \R^{n-j},
$$
such that the graph of $f$
$$
\Gamma(f)=\{(\varphi,T,z,f(\varphi,T,z)): (\varphi,T,z)\in \pi(S)\}\subseteq \pi(S)\times \R^{n-j},
$$
is contained in $S$. Therefore
\begin{equation} \label{eq2}
 \Gamma(f)_{(\varphi,T)}\subseteq S_{(\varphi,T)},
\end{equation}
for every $(\varphi,T)\in \R^{n^2+m}$. Moreover, since $\pi(S)=\pi(\Gamma(f))$ we have 
\begin{equation}\label{eq6}
\pi(S)_{(\varphi,T)}=\pi(\Gamma(f))_{(\varphi,T)},
\end{equation}
for every $(\varphi,T) \in \R^{n^2+m}$.
The function
$$
\begin{array}{cccc}
F:&\pi(S) &\rightarrow  & \Gamma(f)\\
&  (\varphi,T,z) &\mapsto & \left(\varphi,T,z,f(\varphi,T,z)\right)
\end{array}
$$
is definable because its graph is the definable set 
$$
\left\lbrace \left(\varphi,T,z,\varphi,T,z,f(\varphi,T,z)\right): (\varphi,T,z) \in \pi(S) \right\rbrace \subseteq  \pi(S)\times \Gamma(f).
$$
Moreover, $F$ is a bijection with inverse $\pi_{\mid \Gamma(f)}$. Now fix any $(\varphi,T)$, suppose $\pi(S)_{(\varphi,T)}$ is nonempty,
and consider the bijection  $g:\pi(S)_{(\varphi,T)} \rightarrow \Gamma(f)_{(\varphi,T)}$ defined by $g(z)=(z,f(\varphi,T,z))$.
Using the elementary properties we see that $\Gamma(g)$ is definable. Hence, by Lemma \ref{Propdefdim}, we conclude that
\begin{equation} \label{eq3}
 \dim \pi(S)_{(\varphi,T)} = \dim \Gamma(f)_{(\varphi,T)},
\end{equation}
for every $(\varphi,T) \in \R^{n^2+m}$. Note that $\pi(S)_{(\varphi,T)}=\emptyset$ implies $\Gamma(f)_{(\varphi,T)}=\emptyset$, 
and hence (\ref{eq3}) remains true for  $\pi(S)_{(\varphi,T)}=\emptyset$.

Again by the elementary properties, the set
\begin{equation*}
Z'= \left\lbrace  (\varphi,T,x)\in \R^{n^2+m+n} : \varphi \in O_n(\R), \varphi(x) \in  \Gamma(f)_{(\varphi,T)}  \right\rbrace,
\end{equation*}
is definable. 
Note that 
\begin{equation}\label{eq4}
\varphi \left(Z'_{(\varphi,T)} \right)=\Gamma(f)_{(\varphi,T)}
\end{equation}
for every $(\varphi,T) \in O_n(\R)\times \R^{m}$. Moreover, if $\varphi \in \R^{n^2} \setminus O_n(\R)$, we have $Z'_{(\varphi,T)}=\emptyset$ and (\ref{ineqdim}), (\ref{subset}) are satisfied.

Now fix $(\varphi,T) \in O_n(\R)\times \R^{m}$. As $\varphi \in O_n(\R)$ we can apply Lemma \ref{Propdefdim} to get
\begin{equation} \label{eq5}
 \dim Z'_{(\varphi,T)} = \dim \Gamma(f)_{(\varphi,T)}.
\end{equation}
By (\ref{eq1}), (\ref{eq2}) and (\ref{eq4}) we have that 
$$
\varphi\left(Z'_{(\varphi,T)}\right)=\Gamma(f)_{(\varphi,T)}\subseteq S_{(\varphi,T)}=\varphi(Z_T),
$$
and this proves (\ref{subset}). Moreover, since $\pi(S)_{(\varphi,T)}\subseteq \R^j$ and by (\ref{eq3}) and (\ref{eq5}), we have
$$
j \geq \dim \pi(S)_{(\varphi,T)} = \dim Z'_{(\varphi,T)},
$$
that is exactly (\ref{ineqdim}).

We now prove the volume inequality (\ref{ineqhaus}). Let $\Sigma$ be any $j$-dimensional linear subspace of $\R^n$. Fix an orthonormal basis $\{u_1,\ldots , u_j\}$ of $\Sigma$. 
Suppose $\varphi$ is in $O_n(\R)$ and such that $\varphi (u_i)=e_i$ for $i=1, \ldots , j$. Let $\pi_\Sigma$ be the orthogonal projection map from $\R^n$ to $\Sigma$ and $\pit$ the 
projection from $\R^n$ to the coordinate subspace spanned by $e_1, \ldots, e_j$. Note that $\varphi \circ  \pi_\Sigma$ and $\pit \circ \varphi$ coincide on $\Sigma$ and their
kernel is the orthogonal complement $\Sigma^\bot$.
Hence,  $\varphi \circ  \pi_\Sigma=\pit \circ \varphi$. Recalling that $\H^j=\Vol_j$ on $\Sigma$ and $\varphi(\Sigma)$, and using (\ref{eq1}) and Lemma \ref{lemhaus}, we obtain
$$
\Vol_j\left(\pi_{\Sigma}(Z_T)\right) = \Vol_j\left(\varphi \left(\pi_{\Sigma}(Z_T)\right)\right) = 
\mbox{Vol}_j\left(\pit \left(\varphi\left(Z_T\right)\right)\right)=\mbox{Vol}_j\left(\pit\left(S_{(\varphi,T)}\right)\right).
$$ 
Then
\begin{equation}\label{eq7}
V'_j(Z_T)=\sup_{\Sigma} \Vol_j (\pi_\Sigma(Z_T)) \leq \sup_{\varphi \in O_n(\R)}\Vol_j\left(\pit\left(S_{(\varphi,T)}\right)\right).
\end{equation}
Fix $(\varphi,T) \in O_n(\R)\times \R^{m}$. Note that for any set $A\subseteq \R^{n^2+m+n}$ we have $\pit \left(A_{(\varphi,T)}\right)=\{(x_1,\ldots ,x_j ,0 ,\ldots,0):(\varphi,T,x_1 ,\ldots ,x_n)\in A \}$ and 
$\pi(A)_{(\varphi,T)}=\{(x_1,\ldots ,x_j):(\varphi,T,x_1, \ldots ,x_n)\in A\}$.
The latter in conjunction with (\ref{eq6}) gives
$$
\pit \left(S_{(\varphi,T)}\right)=  \pit \left(\Gamma(f)_{(\varphi,T)}\right).
$$
By this and Lemma \ref{lemhaus} we get
\begin{equation}\label{eq8}
\Vol_j\left(\pit \left(S_{(\varphi,T)}\right)\right)= \H^j \left(\pit \left(S_{(\varphi,T)}\right)\right)\leq \H^j\left(\Gamma(f)_{(\varphi,T)}\right).
\end{equation}
Again by (\ref{eq4}) and Lemma \ref{lemhaus} we have
\begin{equation}\label{eq9}
\H^j\left(\Gamma(f)_{(\varphi,T)}\right) = \H^j\left(Z'_{(\varphi,T)}\right),
\end{equation}
for every $(\varphi,T) \in O_n(\R)\times \R^{m}$.
Combining (\ref{eq7}), (\ref{eq8}) and (\ref{eq9}) proves (\ref{ineqhaus}), and thereby completes the proof of Lemma \ref{lemfam}.

\end{proof}

As in Section \ref{sect_measuretheoretic_prelim}, $I$ indicates a nonempty proper subset of $ \{1,\ldots ,n\}$ and $\pi_I$ is the projection map such that 
$\pi_{I}(x_1, \ldots , x_n)= (x_i)_{i \in I}$. 

Applying Lemma \ref{lem1} to the family $Z'$ we conclude that there exist $\M_I$ such that 
$$
\H^j\left(Z'_{(\varphi,T)}\right)\leq  \sum_{|I|=j} \M_I \Vol_j\left(  \pi_I \left(Z'_{(\varphi,T)}\right) \right),
$$ 
for every $(\varphi,T)\in \R^{n^2+m}$.

Let $\pi_{C_I}$ be the orthogonal projection map from $\R^n$ to the coordinate subspace $C_I$ spanned by $e_i$, $i \in I$. We have
$$
\Vol_j\left(  \pi_I \left(Z'_{(\varphi,T)}\right) \right)=\Vol_j\left(  \pi_{C_I} \left(Z'_{(\varphi,T)}\right) \right).
$$
Therefore, recalling (\ref{subset}),
$$
\H^j\left(Z'_{(\varphi,T)}\right) \leq \sum_{|I|=j} \M_I \Vol_j\left(  \pi_{C_I} \left(Z'_{(\varphi,T)}\right) \right)  
\leq \L_Z V_j \left(Z'_{(\varphi,T)}\right)\leq \L_Z V_j \left(Z_{T}\right),
$$
where
$$
\L_Z= \max_j \binom{n}{j} \max_I \M_I.
$$
Finally, combining this with (\ref{ineqhaus}) from Lemma \ref{lemfam}, completes the proof of Proposition \ref{geomineq}.

\section{Proof of Theorem \ref{maintheorem}}\label{Proof of Theorem}
First we assume $Z$ is such that $Z_T=\cl(Z_T)$ for all $T$. By assumption the fibers $Z_T$ are also bounded, and so they are compact. 
Thanks to Lemma \ref{lemdav} we can apply Lemma \ref{applyDavenport} with a Davenport constant $h=M_Z$ depending only on $Z$. 
Then we  use Lemmas \ref{PsiV}, \ref{VI}, and Proposition \ref{geomineq} to bound $V_j(\Psi(Z_T))$,
and this proves the estimate of Theorem \ref{maintheorem} when $Z_T=\cl(Z_T)$. From this special case of the theorem we will deduce the general case.

To this end we first note that
$$\left||\Lambda\cap Z_T|-|\Lambda\cap \cl(Z_T)|\right|\leq |\Lambda\cap \bd(Z_T)|.$$

By Lemma \ref{collintfibers} we see that $C=C(Z)=\{(T,x): x\in \cl(Z_T)\}$ and $B=B(Z)=\{(T,x): x\in \bd(Z_T)\}$ are definable.
Clearly, $C_T=\cl(Z_T)$, and $B_T=\bd(Z_T)$, and these sets are closed and bounded as the sets $Z_T$ are bounded.
Hence, we can apply our theorem with $Z=C$ and then with $Z=B$. For $C$ we obtain
$$\left||\Lambda\cap \cl(Z_T)|-\frac{\Vol(\cl(Z_T))}{\det \Lambda}\right|\leq c_{C}\sum_{j=0}^{n-1}\frac{V_j(\cl(Z_T))}{\lambda_1\cdots \lambda_j}.$$
Note that the constant $c_{C}$ depends only on the family $C$, and thus only on the
family $Z$. Moreover, $\Vol(\cl(Z_T))=\Vol(Z_T)$ by Lemma \ref{measdef} and $V_j(\cl(Z_T))=V_j(Z_T)$ by Lemma \ref{Vjclos}.
Using also $\Vol(\bd(Z_T))=0$ by Lemma \ref{measdef}, and $\bd(Z_T)\subseteq \cl(Z_T)$, we get similarly that
$$ |\Lambda\cap \bd(Z_T)|\leq c_{B}\sum_{j=0}^{n-1}\frac{V_j(Z_T)}{\lambda_1\cdots \lambda_j},$$
again with a constant $c_{B}$ depending only on the family $Z$.
Combining these estimates concludes the proof of Theorem \ref{maintheorem} in the general case.

\section*{Acknowledgements}
It is our pleasure to thank Alessandro Berarducci, Zo\'e Chatzidakis, Marcello Mamino, and Vincenzo Mantova for answering many questions about o-minimal structures.
We thank Francesco Ghiraldin for pointing out the example \cite[Example 2.67]{AmbrosioFuscoPallara} mentioned in the introduction, and Andrea Mondino
for helpful discussions on rectifiability.
We also thank Robert Tichy, Johannes Wallner, and Umberto Zannier for interesting discussions and encouragement.
We are grateful to the referees for providing valuable suggestions that simplified the proof of Lemma \ref{lem1} and improved the exposition of the paper.
Parts of this project have been done while the second author was visiting Graz University of Technology. He is grateful for the 
invitation and the financial support. The first author would like to thank Centro de Giorgi for the hospitality during his visit in Pisa.

\end{document}